\def\ext{\mbox{EXT}}
\def\eso{\mbox{ESO}}
\def\rest{\!\restriction\!}

\def\un{{\rm Un}}
\def\ex{{\rm Ex}}
\def\sort{{\rm sort}}
\def\dom{\rm dom}
\def\ran{\rm ran}
\newcommand{\dep}{\mathcal{D}}
\newcommand{\inc}{\mathcal{I}}
\newcommand{\loo}{{L}_{\omega\omega}}
\newcommand{\looo}{{L}_{\omega_1\omega}}

\newcommand{\lio}{{ L}_{\infty\omega}}
\newcommand{\Loo}{{L}_{\omega\omega}}

\def\phi{\varphi}\def\mm{\mathcal{M}}\def\I{\mbox{I}}\def\II{\mbox{II}}
\newtheorem{theorem}{Theorem}

\newtheorem{definition}[theorem]{Definition}
\newtheorem{lemma}[theorem]{Lemma}
\newtheorem{proposition}[theorem]{Proposition}

\newtheorem{example}[theorem]{Example}

\newcommand{\open}{\Bbb}

\newcommand{\oN}{{\open N}}

\def\mm{\mathfrak{M}}

\def\mn{\mathfrak{N}}
\def\ma{\mathfrak{A}}
\def\mb{\mathfrak{B}}

\def\fraisse{Fra\"\i ss\`e}
\def\ef{Ehrenfeucht-\fraisse}

\def\P{{\cal P}}

\def\D{{\Delta}}

\def\={=\!\!}

\def\qed{$\Box$\medskip}

\def\rob{\mbox{\rm Rob}}
\def\craig{\mbox{\rm Craig}}
\def\Beth{\mbox{\rm Beth}}
\def\sk{\mbox{\rm SK}}
\def\wbeth{\mbox{\rm WBeth}}
\def\sep{\mbox{\rm Sep}}
\newcommand{\pc}[1]{\mathop{\mbox{PC}}(#1)}

\newcommand{\ec}[1]{\mathop{\mbox{EC}}(#1)}

\def\lio{L_{\infty\omega}}

\def\lkp{L_{\kappa}^1}

\def\laa{L(aa)}

\def\sol{L^2}

\def\lcof{L(Q^{\mbox{\tiny cof}}_{\omega})}
\def\lcofc{L(Q^{\mbox{\tiny cof}}_{\le 2^\omega})}

\documentclass{article}
\usepackage{amsfonts,psfrag,hyperref}
\usepackage[pdftex]{graphicx}
\usepackage{latexsym} 
\usepackage{amssymb,amsmath}

\title{Interpolation in   model theory\footnote{This project has received funding from the European Research Council (ERC) under the European Union’s Horizon 2020 research and innovation programme (grant agreement No 101020762).}}

\author{Jouko V\"a\"an\"anen\\
University of Helsinki}

\begin{document}

\maketitle

\begin{abstract}
We bring an abstract model theory perspective to interpolation. We ask,  what is the role of interpolation in the study of extensions of first order logic, such as infinitary logics, generalized quantifiers and 
higher order logics? The abstract model theory approach reveals the basic connections between various interpolation properties in isolation, on their own, as well as  with respect to other model theoretic properties, such as compactness.
\end{abstract}

\tableofcontents

\section{Introduction}

The Craig Interpolation Theorem (\cite{MR104564}) was originally proved by proof theoretic methods. Even today proof theory or methods such as tableaux methods derived from proof theory are the most popular road to interpolation. The model theoretic version of interpolation is the Robinson property (\cite{MR78307}), which can be readily proved with purely model theoretic methods. The statement of  the Robinson property resembles amalgamation which is at the heart of modern model theory. Neither interpolation nor the Robinson property seem to be central in current model theory but the idea of amalgamation is certainly central (see e.g. \cite{MR1221741}). Moreover, the model theoretic proof of the Robinson property is in principle based on the use of one form or another of saturated models and saturated models are certainly important in model theory.

In this chapter we take the axiomatic (or abstract) approach to model theory. We investigate which model theoretic properties of first order logic, which we denote in this chapter by $\loo$, are behind interpolation and the Robinson property. For this to make sense we recall the concept of an abstract logic as well as a variety of extensions of first order logic, some with interpolation and some without. This indicates what  it is in first order (or some other) logic that allows us to derive the interpolation theorem.

The topic of interpolation and its weaker forms in model theoretic languages is a vast topic and we will only touch upon some basic results in this chapter. The topic is so extensive because there are numerous extensions of first order logic as well as numerous weaker forms of interpolation. This leads to a huge  potential for implications and counter-examples. We have chosen to include in this chapter only what we think as illustrative of the situation. A good reference to a more complete picture is still today \cite{MR819531}.

Section \ref{abstract setting} introduces the basic concepts of abstract model theory and the basic general results about compactness, interpolation, the Robinson property and the Beth Definability Theorem (\cite{MR58537}). We prove some basic relationships between these properties. The focus is on how much we have to assume of the logics for the relationships to hold. Section~\ref{Generalized quantifiers} concentrates on generalized quantifiers, wherein we present the basic facts. The amount of material here is huge but we focus on what we think is illuminating. Section~\ref{Infinitary logics} presents the basic facts about interpolation in infinitary logics ending with Shelah's relatively new  logic $L^1_\kappa$, which has an interpolation property (\cite{MR2869022}). Finally in Section~\ref{Higher order logics} we discuss the matter of interpolation in higher order logics, including  logics without negation such as existential second order logic and dependence logic (\cite{MR2351449}).  

We consider both single-sorted and many-sorted logic, as the difference between the two is relevant in the study of interpolation. Everything is many-sorted unless otherwise specified. As is well-known, many-sorted logic can be reduced to single-sorted logic, but this reduction is not so perfect that many-sorted vocabularies would be rendered useless.

The reader is assumed to know basic first order model theory. Otherwise we try to be self-contained. For details on generalized quantifiers, infinitary languages and abstract model theory, a good source is \cite{MR819531}. For unexplained concepts of set theory we refer to \cite{MR1940513}. For basics of model theory we refer to \cite{MR1221741}. 
   
\section{The abstract setting}\label{abstract setting}

By abstract model theory we mean here the study of extensions of first order logic, such as logics with generalized quantifiers, infinitary logics and higher order logics. Interpolation by and large fails in such logics but there are notable exceptions, such as $\looo$ (see Section~\ref{Infinitary logics}), and in some cases interpolation can be established in a relative sense meaning that the interpolant is found, not in the logic under consideration, but in a hopefully not too much bigger logic.

  In its most general form, an {\em abstract logic}, semantically conceived,  is just a triple $L=(S,F,\models)$
where $S$ and $F$ are classes and $\models$ is a subclass of $S\times F$. Elements of  $S$
are called the \emph{structures} of $L$, elements of $F$ are called the
\emph{sentences} of $L$, and the relation $\models$ is called the
satisfaction relation of $L$. Note that we do not assume anything about the ``structures'' that constitute $S$.  They need not be structures in the sense of ordinary first order model theory. They can also be Kripke models or valuations of propositional logic or whatnot. In this generality the concept of an abstract logic covers first order logic and its extensions by generalized quantifiers and infinitary operations, but also modal logic, propositional logic, topological logic \cite{MR560706}, etc. Likewise, the elements of $F$, i.e. the ``sentences'' of $L$, need not be sentences in any usual sense. Normally sentences are identified with finite strings of symbols but sentences of infinitary logics, such as $\looo$ are best thought of as sets (or perhaps trees). Finally, the satisfaction relation $\models$ does not have to have any inductive definition tying together $S$ and $F$. All we assume is that it is a subclass of $S\times F$, which means that it is a class in the sense of set theory\footnote{Nothing essential would change if we worked in class theory instead of set theory.} i.e. it is a definable predicate of set-theory, perhaps from some parameters.

On this level of generality it is, of course, hardly possible to prove any deep results about abstract logics. But if some simple results could be proved, they would automatically apply to first order logic, propositional logic, modal logic, and so on, and would manifest similarity or a ``family resemblance'' of the logics from this point of view, while the logics mentioned live otherwise lightyears from each other. Indeed, as we shall see, some very basic facts about interpolation and model theoretic properties around it can be proved even in our very general setup. Whether the  results that can be proved in this generality are interesting or not, is a reasonable question. Be that as it may, it is however notable that we can perfectly formulate many fundamental model theoretic  \emph{concepts}  on this level of generality, even if no deep results can be proved. 

The basic concepts of abstract logics are the following: A structure $\mm\in S$ is a \emph{model of $\phi\in F$} if $\mm\models\phi$, and a \emph{model of} $\Sigma\subseteq F$, in symbols $\mm\models\Sigma$, if $\mm$ is a model of each $\phi\in \Sigma$. The sentence $\phi\in F$, or a set $\Sigma\subseteq F$, \emph{is consistent} if it has a model.  Whenever $\phi,\psi\in F$, we write $\phi\models\psi$, if for every model $\mm\in S$, $\mm\models\phi$ implies $\mm\models\psi$, and $\phi\equiv\psi$ if $\phi\models\psi$ and $\psi\models\phi$. Whenever $\Sigma\cup\{\psi\}\subseteq F$, we write $\Sigma\models\psi$, if $\mm\models\Sigma$ implies $\mm\models\psi$ for all $\mm\in S$. Two models $\mm,\mm'\in S$ are \emph{$L$-equivalent}, in symbols $\mm\equiv_L\mm'$,  if they are models of the same sentences of $F$. A set $\Sigma\subseteq F$ is called \emph{complete} if for some  $\mm\in S$ we have $\Sigma=\{\phi\in F:\mm\models\phi\}$.  Any consistent set $\Sigma\subseteq F$ can be extended to a complete consistent $\Sigma'$ by taking a model $\mm$ of $\Sigma$ and letting $\Sigma'=\{\phi\in F:\mm\models\phi\}$.

Already this very general approach imposes an equivalence relation $\equiv_L$ on $S$. Different abstract logics may impose different equivalence relations on the same class $S$. Intuitively speaking a stronger logic imposes a finer equivalence relation, if $S$ is fixed. Respectively, the structures in $S$ impose relationships between sentences in $F$, e.g. the relation $\phi\models\psi$.

We can identify a sentence with the class of its models. Likewise, we can identify a structure with the class of sentences it is a model of. Such identifications would lead to further abstraction without necessarily being helpful. But this demonstrates a certain duality between structures and sentences, well-known and much studied in the development of mathematical logic. Of course, these basic relationships between structures and sentences depend heavily on what is our (abstract) logic $L$.

\begin{definition}
We say that an abstract logic $L$ satisfies the \emph{Compactness Theorem} or is \emph{compact}, if  for every set $\Sigma\subseteq F$ the following holds: If  every finite $\Gamma\subseteq\Sigma$ is consistent, then $\Sigma$ itself is consistent. If this holds for all countable $\Sigma$, we say that $L$ satisfies the \emph{Countable Compactness Theorem} or is \emph{countably compact}.
\end{definition}

To formulate interpolation  we introduce the \emph{vocabulary} (class) function  $\tau$ which has $S\cup F$ as its domain.  The values of $\tau$ are arbitrary sets but intuitively $\tau(\phi)$ is  the set of relation, constant, proposition, function, etc -symbols occurring in the sentence $\phi$. Intuitively, $\tau(\mm)$ is  the set of relation, constant, proposition, function, etc -symbols that are interpreted in $\mm$.  We assume always that $\mm\models\phi$ implies $\tau(\phi)\subseteq\tau(\mm)$. If $\Sigma\subseteq F$, we write $\tau(\Sigma)$ for $\bigcup_{\phi\in \Sigma}\tau(\phi)$.  We could go ahead and formulate natural axioms that $\tau$ should satisfy in order to correspond to our intuition about the vocabulary of a sentence or of a structure. However, we skip that here, as it is not relevant for us now. We just call the quadruple $(S,F,\models,\tau)$ an abstract logic. For such a quadruple we can now define the concept of interpolation:

\begin{definition} The abstract logic  $L=(S,F,\models,\tau)$ has the \emph{Interpolation property} if for every $\phi,\psi\in F$, the relation $\phi\models\psi$ implies the existence of $\theta\in F$ such that
\begin{enumerate}
\item $\phi\models\theta$
\item $\theta\models\psi$
\item $\tau(\theta)\subseteq \tau(\phi)\cap \tau(\psi)$.
\end{enumerate}
\end{definition}

In the sense of this definition, first order logic, the infinitary logic $L_{\omega_1\omega}$ (\cite{MR344115}, see section \ref{Infinitary logics} below), Shelah's $L^1_\kappa$ (\cite{MR2869022}, see section \ref{Infinitary logics} below),  single-sorted second order logic (see section~\ref{Higher order logics}), propositional logic, modal logic and topological logic (\cite{MR560706}) all have the interpolation property. Still, considering how many extensions of first order logic have been introduced, see  e.g. \cite{MR819531}, it is a little surprising how few of them have the interpolation property. This circumstance has led to a study of weaker forms of interpolation as well as of the possible reasons why many logics fail to satisfy interpolation.

For the history of the birth of the Interpolation property of first order logic we refer to \cite{MR2438873}.

Assuming the Compactness Theorem, there is an alternative formulation of interpolation:

\begin{definition} The abstract logic  $L=(S,F,\models,\tau)$ has the \emph{Robinson (Consistency) property} (\cite{MR78307}) if for every complete $\Sigma_0\subseteq F$, every  consistent $\Sigma_1\subseteq F$ extending $\Sigma_0$, and every  consistent $\Sigma_2\subseteq F$ extending $\Sigma_0$, such that $\tau(\Sigma_1)\cap\tau(\Sigma_2)=\tau(\Sigma_0)$,  the set $\Sigma_1\cup\Sigma_2$ is consistent.
\end{definition}

The model theoretic proof of the Robinson  property for first order logic is based on the following simple idea: Take an infinite saturated model $\mm_1$ of $\Sigma_1$ and a saturated model $\mm_2$ of $\Sigma_2$ such that $\mm_1$ and $\mm_2$ have the same cardinality. Now $\mm_1\rest\tau(\Sigma_0)$ and $\mm_2\rest\tau(\Sigma_0)$ are saturated elementary equivalent (since $\Sigma_0$ is complete) models of the same cardinality. Hence they are isomorphic. The isomorphism can be used to transfer the interpretations of symbols in $\tau(\Sigma_2)\setminus\tau(\Sigma_0)$ from $\mm_2$ to $\mm_1$. This yields an expansion $\mn$ of $\mm_1$ such that $\mn\rest\tau(\Sigma_1)=\mm_1$ and $\mn\rest\tau(\Sigma_2)\cong\mm_2$. Thus $\mn$ is a model of $\Sigma_1\cup\Sigma_2$. It should be noted that this proof is an overkill but worth knowing because it is probably the quickest proof. Indeed, saturated models are only known to exist if we assume some amount of the Generalized Continuum Hypothesis or alternatively the existence of strongly inaccessible cardinals, neither of which is really needed for the result.  The use of saturated models can be avoided by  more refined methods, for example special models (\cite{MR1059055}) or recursively saturated models (\cite{MR403952}), which always exist. Alternatively, we present below in Theorem~\ref{lindstr} a gentler proof due to Lindstr\"om. We will also prove below in Theorem~\ref{5} the Robinson property from compactness and interpolation in a very general setting.

To connect the Interpolation property and the Robinson  property we need to assume a little bit about abstract logics. To this end we define:

\begin{definition}\begin{enumerate}
\item The abstract logic $L=(S,F,\models,\tau)$ is \emph{closed under conjunction} if for every $\phi,\psi\in F$ there
 is $\theta\in F$, denoted $\phi\wedge\psi$,  such that $\tau(\theta)=\tau(\phi)\cup\tau(\psi)$ and for all $\mm\in S$:
$$\mm\models\phi\wedge\psi\iff\mm\models\phi\mbox{ and }\mm\models\psi.$$
\item $L$ is \emph{closed under negation} if for every $\phi\in F$ there is $\theta\in F$, denoted $\neg\phi$, such that $\tau(\theta)=\tau(\phi)$ and for all $\mm\in S$:
$$\mm\models\neg\phi\iff\mm\not\models\phi.$$ 
\item If $L$ is closed under both conjunction and negation, we denote $\neg(\neg\phi\wedge\neg\psi)$ by $\phi\vee\psi$ and $(\phi\wedge\psi)\vee(\neg\phi\wedge\neg\psi)$ by $\phi\leftrightarrow\psi$.
\end{enumerate}
\end{definition}

\begin{theorem}\label{5}
If $L$ is closed under negation and conjunction, and satisfies the Compactness Theorem, then the following conditions are equivalent:
\begin{enumerate}
\item $L$ has the Interpolation property.
\item $L$ has the Robinson  property.
\end{enumerate}
\end{theorem}

Rather than proving this, we formulate and prove a stronger version.
Since in abstract model theory there are numerous examples of the failure of the Interpolation property as well as of the Robinson property, it makes sense to introduce relative versions of both. To this end we need the concept of a \emph{sublogic}:

\begin{definition}\begin{enumerate}
\item A logic $L=(S,F,\models,\tau)$ is a sublogic of $L'=(S',F',\models',\tau')$, in symbols $L\le L'$,  if $S=S'$ and for every $\phi\in F$ there is $\phi'\in F'$ such that $\phi\equiv\phi'$ i.e. for all $\mm\in S$, $\mm\models\phi$ if and only if $\mm\models'\psi$, and in addition, $\tau(\phi)=\tau'(\phi')$. We usually identify $\phi$ and $\phi'$ and thereby $F\subseteq F'$. 
\item Logics $L$ and $L'$ are \emph{equivalent}, $L\equiv L'$, if $L\le L'$ and $L'\le L$.
\end{enumerate}

\end{definition}
Note that here both logics in the above definition have the same class $S$ of structures. The more general case of different classes of structures can be dealt with by means of the so-called Chu-transform, see \cite{MR4357456}. Note that a sublogic of a (countably) compact logic is (countably) compact.

\begin{definition}  Suppose $L=(S,F,\models,\tau)$ and $L'=(S',F',\models',\tau')$ are abstract logics such that $L\le L'$. 
\begin{enumerate}
\item $\craig(L,L')$ holds if for every $\phi,\psi\in F$, the relation $\phi\models\psi$ implies the existence of $\theta\in F'$ such that
for all $\mm\in S$, $\mm\models\phi\Rightarrow\mm\models'\theta$, $\mm\models'\theta\Rightarrow\mm\models\psi$, and  $\tau(\theta)\subseteq \tau(\phi)\cap \tau(\psi)$. By $\craig(L)$ we mean $\craig(L,L)$.

\item $\rob(L,L')$ holds if for every complete $\Sigma_0\subseteq F'$, every   $\Sigma_1\subseteq F$, such that $\Sigma_0\cup \Sigma_1$ is consistent, and every    $\Sigma_2\subseteq F$, such that  $\Sigma_0\cup\Sigma_2$ is consistent, if $\tau(\Sigma_0\cup\Sigma_1)\cap\tau(\Sigma_0\cup\Sigma_2)=\tau(\Sigma_0)$,  then the set $(\Sigma_0\cap F)\cup\Sigma_1\cup\Sigma_2$ is consistent. By $\rob(L)$ we mean $\rob(L,L)$.

\end{enumerate}\end{definition}

\begin{theorem}
If $L\le L'$ are closed under negation and conjunction, and $L'$ satisfies the Compactness Theorem, then the following are equivalent:
\begin{enumerate}
\item \craig($L,L'$).
\item \rob($L,L'$).
\end{enumerate}
\end{theorem}

\begin{proof} We follow the standard proof, as e.g. in \cite[Chapter II, 7.1.5.]{MR819531}.

(1) implies (2): Suppose a complete $\Sigma_0\subseteq F'$ is given. Suppose also   $\Sigma_1\subseteq F$, such that $\Sigma_0\cup \Sigma_1$ is (wlog) complete, and    $\Sigma_2\subseteq F$, such that  $\Sigma_0\cup\Sigma_2$ is (wlog) complete, are given, and furthermore  $\tau(\Sigma_0\cup\Sigma_1)\cap\tau(\Sigma_0\cup\Sigma_2)=\tau(\Sigma_0)$. Finally, suppose the set $(\Sigma_0\cap F)\cup\Sigma_1\cup\Sigma_2$ is inconsistent.
By the Compactness Theorem of $L'$ there are some $\{\eta_0,\ldots,\eta_k\}\subseteq\Sigma_0\cap F$, $\{\phi_1,\ldots,\phi_n\}\subseteq \Sigma_1$ and $\{\psi_1,\ldots,\psi_m\}\subseteq \Sigma_2$ such that $\{\eta_0,\ldots,\eta_k\}\cup\{\phi_1,\ldots,\phi_n\}\cup\{\psi_1,\ldots,\psi_m\}$ is inconsistent. 
Since $\Sigma_0\cup\Sigma_1$ and $\Sigma_0\cup\Sigma_2$ are individually both consistent, we have $n>0$ and $m>0$. Let $\eta=\eta_1\wedge\ldots\wedge\eta_n$, $\phi=\phi_1\wedge\ldots\wedge\phi_n$ and $\psi=\neg(\psi_1\wedge\ldots\wedge\psi_m)$. Thus $\eta\wedge\phi\models\psi$. Let $\theta\in F$ such that  $\eta\wedge\phi\models\theta$,
 $\theta\models\psi$
and  $\tau(\theta)\subseteq \tau(\eta\wedge\phi)\cap \tau(\psi)$. Since $\Sigma_0\cup\Sigma_1$ is complete, $\theta\in\Sigma_0\cup\Sigma_1$. Since $\Sigma_0\cup\Sigma_2$ is complete and consistent, $\neg\theta\in\Sigma_0\cup\Sigma_2$. Since $\Sigma_0$ is complete, $\{\theta,\neg\theta\}\subseteq\Sigma_0$, contrary to the consistency of $\Sigma_0$. 

(2) implies (1): Suppose  $\phi\models\psi$, where $\phi,\psi\in F$ and $\tau_0=\tau(\phi)\cap \tau(\psi)$. Let $\Sigma_0$ be the set of $\theta\in F'$ such that $\phi\models\theta$ and $\tau(\theta)\subseteq \tau_0$. We show now that $\Sigma_0\models\psi.$ 
Otherwise there is a model $\mm$ of $\Sigma_0\cup\{\neg\psi\}$. 
Let $\Sigma_0^*$ be the set of $\theta\in F'$ such that $\mm\models\theta$ and $\tau(\theta)\subseteq\tau_0$. Note that $\Sigma^*_0$ is complete. The set $\Sigma_0^*\cup\{\psi\}$ is consistent for otherwise the Compactness Theorem gives $\{\theta_1,\ldots,\theta_k\}\subseteq \Sigma_0^*$ such that $\psi\models\neg(\theta_1\wedge\ldots\wedge\theta_k)$ implying  $\neg(\theta_1\wedge\ldots\wedge\theta_k)\in \Sigma_0$, a contradiction. 
Now both $\Sigma^*_0\cup\{\neg\psi\}$ and $\Sigma_0^*\cup\{\psi\}$ are consistent. By (2), 
$(\Sigma^*_0\cap F)\cup\{\phi\}\cup\{\neg\psi\}$ is consistent, a contradiction. Having now proved $\Sigma_0\models\psi$ we use again the Compactness Theorem. We obtain $\{\theta_1,\ldots,\theta_k\}\subseteq \Sigma_0$ such that $\{\theta_1,\ldots,\theta_k\}\models\psi$. Letting $\theta=\theta_1\wedge\ldots\wedge\theta_k$ yields $\theta\in F'$, $\phi\models\theta$, $\theta\models\psi$ and $\tau(\theta)\subseteq\tau(\phi)\cap\tau(\psi)$.
 \end{proof}

In the above theorem it is enough to assume that $L'$ is countably compact, provided that $L'$ satisfies the condition that there are only countably many sentences with a given countable vocabulary and the property \rob($L,L'$) is formulated for countable vocabularies only.

\begin{definition} An abstract logic $(S,F,\models,\tau)$ is \emph{classical}, if:
 \begin{enumerate}

\item $\tau(\phi)$ consists, for all $\phi\in F$, of relation, constant and function symbols.
\item $S$ is a class of structures $\mm$ in the sense of first order logic and $\tau(\mm)$ has its usual meaning as the vocabulary of the structure $\mm$. 
\item If $\mm\in S$ and $a_1,\ldots,a_n\in M$, we assume that the structure $(\mm,a_1,\ldots,a_n)$, obtained from $\mm$ by distinguishing the elements $a_1,\ldots,a_n$ as interpretations of new constant symbols $c_1,\ldots,c_n$ is also in $S$. 
\item If $\sigma$ is a subvocabulary of $\tau(\mm)$, then the \emph{reduct} $\mm\rest\sigma$ is in $S$. 
\item If $\mm\in S$ and $\phi\in F$, then  $\mm\models\phi$ iff $\mm\rest\tau(\phi)\models\phi$. 
\item  $L$ satisfies renaming\footnote{Essentially, renaming says that we can change symbols in a formula and truth is preserved, if we make the respective changes in structures.} in the sense of Definition 1.1.1 of \cite{MR819531}.  
\item A \emph{formula\footnote{We use constant symbols to play the role of variables.}} $\phi(x_1,\ldots,x_n)$ of $L$ is a sentence $\phi(c_1,\ldots,c_n)$, where $c_1,\ldots,c_n$ are new constant symbols.
\end{enumerate}
 A classical logic $L$ is \emph{atomic} if for every relation symbol $R(x_1,\ldots,x_n)$ and constant symbols $c_1,\ldots,c_n$ there is a sentence $\phi(c_1,\ldots,c_n)\in F$, such that $\tau(\phi(c_1,\ldots,c_n))=\{R,c_1,\ldots,c_n\}$ and for all $\mm\in S$ with $\tau(\phi)\subseteq\tau(\mm)$, and all $a_1,\ldots,a_n\in M$.
$$(\mm,a_1,\ldots,a_n)\models\phi(c_1,\ldots,c_n)\iff (a_1,\ldots,a_n)\in R^\mm.$$ We write $\phi(c_1,\ldots,c_n)$ as $R(c_1,\ldots,c_n)$.

\end{definition}

Apart from ordinary first order logic and its extension by generalized quantifiers, infinitary logical operations, and higher order quantifiers,  first order logic with finite models is a classical atomic abstract logic. Also first order logic with ordered models, i.e. models with a distinguished linear order $<$, which is an element of every vocabulary considered, and does not disappear in a reduct, is a classical atomic abstract logic. The two-variable fragment of first order logic, where vocabularies consist only of binary and unary predicates symbols, constant symbols and contain no function symbols, and only two variables are used overall, is a classical atomic abstract logic.

A stronger but still natural and prevalent property of logics is \emph{regularity} in the sense of \cite[Chapter II]{MR819531}.  

For classical logics it makes sense to talk about the cardinality of the models. In particular, we can formulate the important \emph{L\"owenheim property}: Every sentence with a model has a countable model.

An $\ec{L}$-class (``E'' for ``elementary'')  is a subclass $K$ of $S$ such that  for some $\phi\in F$, $\mm\in K$ if and only if $\mm\models\phi$. We then write $\tau(K)=\tau(\phi)$. A $\pc{L}$-class (``P'' for ``projective'')  is a subclass $K$ of $S$ such that $\tau(\mm)$ is 
a fixed $\tau$ for $\mm\in K$, and for some $\phi\in F$, 
$\mm\in K$ if and only if $\mn\models\phi$ for some $\mn$ such that $\mn\rest\tau=\mm$\footnote{For example, the class $K$ of infinite models of the empty vocabulary is a $\pc{\loo}$-class as we can let $\phi$ be the first order sentence of vocabulary $\{<\}$ which says that $<$ is a linear order without last element.}. We then write $\tau(K)=\tau$. 
Here $\tau(\phi)$ may have more sorts than $\tau$\footnote{This will be relevant when we discuss second order logic.}. If no new sort occur in $\tau(\phi)$, then $\pc{\loo}$-definability coincides with existential second order definability. We may consider the family of all $\pc{L}$-classes an abstract logic in the obvious sense.

Craig \cite{MR104565} showed that his interpolation
theorem
  has an equivalent formulation as a separation
property:

\begin{definition}
 The  \emph{separation property} $\sep(L,L')$ holds if any two disjoint $\pc{L}$-classes $K_0$ and $K_1$ with the same vocabulary $\tau$ can be separated by an
$\ec{L'}$-class $K$, i.e. $K_0\subseteq K$ and $K\cap K_1=\emptyset$ with $\tau(K)=\tau$. A logic $L$ satisfies the \emph{separation property} if it satisfies $\sep(L,L)$. Two logics $L$ and $L'$ satisfy the \emph{Souslin-Kleene interpolation 
property}, $\sk(L,L')$,  if every $\pc{L}$-class $K$ whose complement in the class $\{\mm\in S:\tau(\mm)=\tau(K)\}$ is also a $\pc{L}$-class, is  an $\ec{L'}$-class.
\end{definition}

Naturally $\sep(L,L')$ implies $\sk(L,L')$, but not conversely (\cite{MR437336}). Note that $\sep(L,L')$ implies $L\le L'$, if $L$ is closed under negation. The family of 
all $\pc{L}$-classes can be construed as a logic itself, as is carefully explained in  (\cite{MR457146}). This logic, denoted $\Delta(L)$, has similar regularity properties as $L$. For example, it is classical, atomic and closed under conjunction and negation if $L$ is. It is the smallest extension of a regular abstract logic $L$ to a regular abstract logic with the Souslin-Kleene Interpolation property
(\cite{MR457146}). In fact, the Souslin-Kleene Interpolation property
is also called the $\Delta$-interpolation property (\cite{chapter:predicate}). If we limit ourselves to $\pc{L}$-classes that do not add new sorts (i.e. the defining formula does not have new sorts, only new predicates, functions and constants) we obtain a more restrictive extension $\Delta^1_1(L)$ which behaves very much in the same way as $\Delta(L)$. Of course, $\Delta^1_1(L)\subseteq\Delta(L)$.

\begin{example}
The class of models $(M,E)$, where $E\subseteq M\times M$ is an equivalence relation with uncountably many uncountable classes, is definable in $\Delta(L(Q_1))$ but not in $L(Q_1)$ (see Theorem~\ref{17}, especially its proof). Hence $\Delta(L(Q_1))$ is a proper extension of $L(Q_1)$.
\end{example}

We prove the following simple result, generalizing a result from \cite{MR104565}, in some detail because we assume in a sense the minimal amount of the logics, so it is not so clear that the classical proof works. It is perhaps interesting to see where the different assumptions are used.

\begin{proposition}
Assume $L$ and $L'$ are classical and satisfy $S=S'$. Suppose also that  $L$ is closed under negation. 
The following conditions are equivalent:
\begin{enumerate}
\item $\craig(L,L')$.
\item $\sep(L,L')$
\end{enumerate}
\end{proposition}

\begin{proof}
Assume (1). Suppose $K_0$ and $K_1$ are disjoint $\pc{L}$-classes and $\tau=\tau(K_0)=\tau(K_1)$. Let $\phi_0\in F$ with 
$\mm\in K_0$ if and only if $\mn\models\phi_0$ for some $\mn$ such that $\mn\rest\tau=\mm$ and $\phi_1\in F$ with 
$\mm\in K_1$ if and only if $\mn\models\phi_1$ for some $\mn$ such that $\mn\rest\tau=\mm$. Using the renaming property of $L$ we can change the non-logical symbols of $\phi_1$, other than those in $\tau$, to something completely new, and thereby make sure $\tau=\tau(\phi_0)\cap\tau(\phi_1)$. Using closure of $L$ under negation, we can find $\neg\phi_1\in F$. Since $K_0$ and $K_1$ are disjoint, $\phi_0\models\neg\phi_1$. By (1), there is $\theta\in F'$ such that $\phi_0\models\theta$, $\theta\models\phi_1$ and $\tau(\theta)\subseteq\tau$. Now the $\ec{L'}$-class $K=\{\mm\in S:\mm\models\theta\}$ separates $K_0$ and $K_1$ i.e. $K_0\subseteq K$ and $K\cap K_1=\emptyset$.

Assume  (2). Suppose $\phi\models\psi$, where $\phi,\psi\in F$. Let $\tau=\tau(\phi)\cap\tau(\psi)$. Since $L$ is closed under negation, $\neg\psi\in F$. Let $K_0=\{\mm\rest\tau:\mm\in S, \mm\models\phi\}$ and $K_1=\{\mm\rest\tau:\mm\in S, \mm\models\neg\psi\}$. Now $K_0\cap K_1=\emptyset$. By (2) there is an $\ec{L'}$-class $K$ such that $K_0\subseteq K$, 
$K\cap K_1=\emptyset$ and $\tau(K)=\tau$. Let $\theta\in F'$ such that $\tau(\theta)=\tau$ and $K=\{\mm\in S:\tau(\mm)=\tau\mbox{ and }\mm\models\theta\}$. Now $\phi\models\theta$, $\theta\models\psi$, and
$\tau(\theta)\subseteq\tau(\phi)\cap\tau(\psi)$. 
\end{proof}

The assumption about being closed under negation is essential as the following example from \cite{MR2134728} shows: Let $L$ be first order logic, with negation only in front of atomic formulas, added with the generalized quantifier ``there are infinitely many $x$ such that ...'', i.e. the quantifier $Q_0$.  Every $\pc{L}$-class is $\pc{\loo}$, as can be proved by induction on $\pc{L}$-definitions, and therefore separation holds for $L$, even $\sep(L,\loo)$. On the other hand, the proof of Theorem~\ref{17}  shows that $\craig(L)$ fails.

A famous consequence of the Craig Interpolation Theorem is the Beth Definability Theorem  (\cite{MR58537}), which can be formulated as follows in the abstract logic setting:

\begin{definition}Suppose $L$ and $L'$ are classical, atomic and closed under conjunction, and negation.
\begin{enumerate}
\item  We say that $\Beth(L,L')$ holds if  every $\phi\in F$ and every predicate symbol $P\in\tau(\phi)$ of arity $n$ satisfy: If $\phi'$ denotes $\phi$ with $P$ renamed as $P'\notin\tau(\phi)$  and $\phi\wedge\phi'\models P(c_1,\ldots,c_n)\leftrightarrow P'(c_1,\ldots,c_n)$, where $c_1,\ldots,c_n\notin\tau(\phi)$, then
 there is $\theta\in F'$ such that  $P\notin\tau(\theta)$ and $\phi\models P(c_1,\ldots,c_n)\leftrightarrow\theta$. A logic $L$ has the \emph{Beth property}, $\Beth(L)$, if $\Beth(L,L)$ holds. 
 \item We say that $\wbeth(L,L')$ holds if the above condition for $\Beth(L,L')$ holds with the additional assumption that every $\mm\in S$ with $\tau(\mm)=\tau(\phi)\setminus\{P\}$ can be expanded to a model $\mn\in S$ such that $\mn\models\phi$.
A logic $L$ has the \emph{weak Beth property} (\cite{MR317923}), $\wbeth(L)$, if $\wbeth(L,L)$ holds.
\end{enumerate}\end{definition}

Clearly, $\Beth(L,L')$ implies $\wbeth(L,L')$. Note also, that $\wbeth(L,L')$ implies $L\le L'$.

It is interesting to note that  the fact that first order logic has the Beth property, was proved by Beth well before the Craig Interpolation Theorem was published. The Beth property is sometimes formulated with a theory in place of our single sentence $\phi$.  The two formulations are equivalent for $\loo$, thanks to the Compactness Property.

The following establishes an easy connection between interpolation and the Beth property:

\begin{proposition}\label{cib}
Suppose $L$ and $L'$ are classical abstract logics, closed under conjunction and negation, and $L\le L'$. Then $\craig(L,L')$ implies $\Beth(L,L')$.
\end{proposition}

\begin{proof}Suppose $\phi\in F$ such that 
 if $\phi'$ denotes $\phi$ with $P$ renamed as $P'$, which is not in $\tau(\phi)$, then $\phi\wedge\phi'\models P(c_1,\ldots,c_n)\leftrightarrow P'(c_1,\ldots,c_n)$.
Thus
$\phi\wedge P(c_1,\ldots,c_n)\models\phi'\to P'(c_1,\ldots,c_n).$ By $\craig(L,L')$ there is $\theta\in F'$  such that 
$\phi\wedge P(c_1,\ldots,c_n)\models\theta$, $\theta\models\phi'\to P'(c_1,\ldots,c_n),$ and $\tau(\theta)=(\tau(\phi)\setminus\{P\})\cup\{c_1,\ldots,c_n\}$. Then
$\phi\models P(c_1,\ldots,c_n)\leftrightarrow\theta$. \end{proof}

 Even compactness does not help to prove the converse: The logic $\lcofc$ (see section~\ref{Generalized quantifiers}) has an extension (the so-called ``Beth-closure'' of $\lcofc$) which has the Beth property and is compact but it does not satisfy interpolation, not even Souslin-Kleene interpolation \cite{MR808815}. However, there is a stronger form of the Beth property, \emph{projective} $\Beth(L)$, which is actually equivalent to $\craig(L)$ for regular logics (\cite[p. 76]{MR819531}). Interpolation cannot be weakened to Souslin-Kleene interpolation in Proposition~\ref{cib}, as the $\Delta$-closure of $\lcofc$ does not have the Beth
property \cite{MR546916,MR1045373}.

\begin{definition}\begin{enumerate}
\item A classical abstract logic $L=(S,F,\models,\tau)$ is \emph{fully classical} if $S$ consists of \emph{all} structures of first order logic.
 
 \item  The logic $L$ satisfies \emph{relativization} if for every $\phi\in F$ and every formula $\psi(x)\in F$ there is a sentence $\theta\in F$ such that for all structures $\mm$
 we have $\mm\models\theta$ if and only if the relativization\footnote{The relativization $\mm^{(A)}$ of $\mm$ to $A$
 is the structure which has $A$ as the domain and interpretations of the non-logical symbols as follows: The $n$-ary relation symbols $R$ in the vocabulary of $\mm$ are interpreted as $R^\mm\cap A^n$. The $n$-ary function symbols $f$ are interpreted as restrictions of $f^\mm$ to $A^n$. Finally, the constant symbols are interpreted in the same way as in $\mm$.} $\mm^{(A)}$ of $\mm$ to $A$ satisfies $\phi$, where $A$ is the set of $a\in M$ such that $\mm\models\psi(a)$, and it is assumed that $A$ is closed under the functions of $\mm$ and also contains the constants of $\mm$.
\item The classical logic $L$ satisfies \emph{isomorphism closure}, if $\mm\models\phi\iff\mn\models\phi$ whenever $\mm,\mn\in S$, $\mm\cong\mn$, and $\phi\in F$.
\end{enumerate}
\end{definition}

Examples of fully classical  abstract logics satisfying relativization and  isomorphism closure are  first order logic, $\looo$, second order logic as well as many  extensions of first order logic by generalized quantifiers.

\begin{theorem}[\cite{MR819550}]\label{13}
Assume that the abstract logic $L$ is fully classical, satisfies both relativization and isomorphism closure, and has the property that the vocabulary of each sentence is finite.\footnote{This is called the ``Finite Occurrence property''. It essentially means that each sentence is a finite string a symbols. It follow easily from compactness.}
If $L$ has the (many-sorted) Robinson property then it  satisfies the Compactness Theorem.
\end{theorem}

\begin{proof}We present a rough sketch only and refer to \cite{MR819550} for  a detailed proof. We indicate the main idea of proving countable compactness which should give a good idea of the proof in the general case.
Suppose $T=\{\phi_n:n<\omega\}$ is a counter-example to countable compactness. For each $n$ there is $\mm_n\models\{\phi_m:m<n\}$. W.l.o.g. the domains of the models $\mm_n$ are disjoint and the structures $\mm_n$ are relational. 
Let $\ma$ be the disjoint union of the models $\mm_n$, $n<\omega$. We expand $\ma$ with a new sort $s$ consisting of a copy  of the natural numbers with their natural order $<$ as well as a function $f$ which maps each element of $M_n$ to the  $n$'th natural number in the order $<$. Let $\ma'$ be the expansion. Now in any model $\mb$ that is $L$-equivalent to $\ma'$ the order-type of $<^\mb$ must be $\omega$, for if there were a non-standard number $b$ in the sort $s$ part of $B$, its pre-image $(f^\mb)^{-1}(b)$ would determine a model in which each $\phi_n$ is true, contrary to the inconsistency of $T$.  
Let $\tau_0$ consist of the sort $\{s\}$ and unary predicates $P_n$, $n<\omega$, of sort $s$. Let $\ma^*$ be the expansion of $\ma'$ by letting $P_n^{\ma^*}$ consist of the first $n$ elements of $<$ for all $n<\omega$. Let $\ma''$ be a new $\tau_0\cup\{c\}$-structure, where $c$ is a new constant of sort $s$, which has $\omega+1$ as its domain and $P_n^{\ma''}=\{0,\ldots,n-1\}$ as well as $c^{\ma''}=\omega$. By  isomorphism closure and the finite occurrence assumption, $\ma^*\rest\tau_0\equiv_L\ma''\rest\tau_0$. By the Robinson property, there is $\mb$ such that $\mb\rest\tau(\ma^*)\equiv_L\ma^*$ and $\mb\rest\tau(\ma'')\equiv_L\ma''$. This contradicts the inconsistency of $T$, as $(f^\mb)^{-1}(c^\mb)$ gives rise to a model of each $\phi_n$.
 \end{proof}

The assumption of the finite occurrence property in Theorem~\ref{13} can be considerably weakened, see \cite{MR819550}.

To present an elementary proof of the Robinson property of  first order logic, we recall the following game, due to A. Ehrenfeucht \cite{MR126370}: Let
$L$ be a finite relational vocabulary and $\ma, \mb$
$L$-structures
  such that $A \cap B = \emptyset$. We use $\mbox{EF}_n(\ma,\mb)$
   to denote the $n$-move \ef\ game on
$\ma$ and $\mb$. During each round of the game player $\I$ first
picks an element from one of the models, and then player $\II$
picks an element from the other model. In this way a relation
$p=\{(a_1,b_1),\ldots,(a_n,b_n)\}\subseteq A\times B$ is
built. If $p$ is a partial isomorphism between $\ma$ and $\mb$,
player $\II$ is the winner of this play. Player $\II$ has a
winning strategy in this game if and only if the models $\ma$ and
$\mb$ satisfy the same first order sentences of quantifier rank
at most $n$. If player $\II$ has a winning strategy $\tau$ in
$\mbox{EF}_n(\ma,\mb)$, the set $I_i$ of sets of
 positions
$
\{(a_1,b_1),\ldots,(a_{j},b_{j})\}\subseteq A\times B$
which can be continued to a position
$
\{(a_1,b_1),\ldots,(a_{n-i},b_{n-i})\}\subseteq A\times
B$ in which player $\II$ has used $\tau$, form an
increasing chain $I_n\subseteq I_{n-1}\subseteq\ldots\subseteq I_0$
known as a {\em back-and-forth sequence},  introduced by  R.
\fraisse\ \cite{MR79746}. The name derives from the fact that if
$p\in I_{i+1}$, $|p|=k$, and $a\in M$ (or $b\in M'$), then there is $b\in
M'$ (respectively, $a\in M$) such that $p\cup\{(a,b)\}\in I_i$. In this case 
we say that $I_{i+1}$ satisfies the $k$-back-and-forth condition w.r.t. $I_i$ 
for $k$-sequences between the models $\ma$ and $\mb$. 
Conversely, if a back-and-forth sequence $I_n\subseteq
I_{n-1}\subseteq\ldots\subseteq I_0$ exists, then player $\II$ can
use it to win the game $\mbox{EF}_n(\ma,\mb)$. So back-and-forth
sequences and winning strategies of $\II$ go hand in hand. This
explains why the game is generally called the
Ehrenfeucht-\fraisse\ game. The infinite version, where players play $\omega$ moves, is denoted 
$\mbox{EF}(\ma,\mb)$.

The following proof of the Robinson property for $\loo$, using the
back-and-forth method just described, is due to Per  Lindstr\"om (see \cite{MR3409279}). This
is the argument that became the proof of the so-called Lindstr\"om's Theorem (\cite{MR244013}).
%
%
%

\begin{theorem}[\cite{MR244013}]\label{lindstr}Suppose $L$ is a regular abstract logic in the sense of \cite[Chapter II]{MR819531} such that  $L$ is  compact, has the L\"owenheim property,  and has the property that the vocabulary of each sentence is finite.\footnote{This actually follows from compactness with an easy argument.} Then $L$ has the Robinson property and, in fact, $L\equiv\loo.$
\end{theorem}

\begin{proof}
In the Robinson property, suppose $\Sigma_1$ and $\Sigma_2$ are consistent extensions of a complete theory $\Sigma_0$. Let $\tau_1$ be the vocabulary of $\Sigma_1$, $\tau_2$ that of $\Sigma_2$, and $\tau$ that of $\Sigma_0$, where $\tau=\tau_1\cap\tau_2$. By
assumption, there is a model $\mm_1$ of $\Sigma_1$ and a model $\mm_2$ of
$\Sigma_2$ and then  $\mm_1\rest \tau\equiv\mm_2\rest \tau$. Thus there is, for any $n\in\oN$,  a
back-and-forth sequence $(I_i:i\le n)$ for $\mm_1\rest \tau$ and
$\mm_2\rest \tau$. Let $\tau'_2$ be a copy of $\tau_2$ such that  $\tau_2\cap
\tau'_2=\emptyset$. Let $\tau'$ be the vocabulary resulting from $\tau$ in
this translation. Let $\mm_2'$ be the translation of $\mm_2$ to
the vocabulary $\tau'_2$. Let $R$ be a new unary and $<$ a new binary predicate symbol.
Let $\Gamma$ be the set of first order sentences
which state:

\begin{enumerate}
\item The complete first order $\tau_1$-theory of $\mm_1$.

\item The complete first order $\tau'_2$-theory of $\mm_2$.

\item $(R,<)$  is a non-empty linear order in which
every element with a predecessor has an immediate predecessor.

\item First order sentences which state, by means of   new
$2n+1$-ary predicates $I^n$, that there is, for each $n<\omega$,  an $n$-back-and-forth sequence 
$I^n(i,\cdot,\cdot)$, parametrized by $i\in R$, between the
$\tau$-reduct and the $\tau'$-reduct of the universe, meaning that if $i-1$ denotes the predecessor (if it exists) of $i$ in $R$, then $I^n(i,\cdot,\cdot)$ satisfies the $n$-back-and-forth condition w.r.t. $I^{n+1}(i-1,\cdot,\cdot)$ for $n$-sequences between the $\tau$-reduct and the $\tau'$-reduct of the universe.

\end{enumerate}

For all $n\in\oN$ there is a model of $\Gamma$ with $(R,<)$ of length
$n$. By the Compactness Theorem, there is a  model $\mn$ of $\Gamma$
with $(R,<)$ non-well-founded.  By the L\"owenheim property we may
assume $\mn$ is countable.  Let $\mn_1=\mn\rest \tau_1$ and
$\mn_2=\mn\rest \tau'_2$. Since  $(R,<)^\mn$ is
non-well-founded, player $\II$ has a winning strategy for the
infinite game $\mbox{EF}(\mn_1,\mn_2)$. Since $\mn_1$ and $\mn_2$ are
countable,  $\mn_1\cong\mn_2$. This implies that $\mn_1$ can be
expanded to a model of $\Sigma_1 \cup \Sigma_2$. This ends the proof of Robinson property.

Lindstr\"om observed that the above argument works for any logic
which is countably compact and has  the L\"owenheim
property. But are there such extensions of $\loo$? To see why
there are none,
%
%
let us write $\mm\sim_n\mm'$ if player $\II$ has a winning
strategy in $\mbox{EF}_n(\ma,\mb)$. This equivalence relation divides the
class of all models with a finite vocabulary $L$  into a finite number of
equivalence classes $C$, each definable by a first order sentence of
quantifier rank at most $n$, namely the conjunction of all first order sentence of quantifier rank at most $n$
that are true in one (hence all) models in $C$. 
Suppose now $\phi\in F$ is not first order definable, i.e. there is no first order $\psi$ such that $\phi\equiv\psi$ and $\tau(\phi)=\tau(\psi)$. By our assumption, $\tau(\phi)$ is finite. There is no $n$ such that the class of models of $\phi$ is closed under $\sim_n$. In other words, for all $n$ there are models $\mm_n$ and $\mn_n$ such that $\mm_n\sim_n\mn_n$, $\mm_n\models\phi$ and $\mn_n\models\neg\phi$. By countable compactness there are elementarily equivalent $\mm$ and $\mn$ such that $\mm\models\phi$ and $\mm\models\neg\phi$. Now we can continue as above and obtain two isomorphic models, one of $\phi$ and the other of $\neg\phi$, a contradiction.
\end{proof}

An amusing corollary of the above Lindstr\"om Theorem is that $\loo$ satisfies the Souslin-Kleene Interpolation theorem: Consider $\Delta(\loo)$. It is a regular logic, compact and has the L\"owenheim property. Hence $\Delta(\loo)\equiv\loo$.
There is no similarly easy way (as far as is known) to deduce the Beth property or the Craig Interpolation property directly from 
Lindstr\"om's Theorem. 
 
The proof of Theorem~\ref{lindstr} actually gives the following slightly more general result:

\begin{theorem}[\cite{MR819534}, see also \cite{MR2134728}]
Suppose $L$ is a regular abstract logic, except that $L$ is \emph{not} assumed to be closed under negation. If $L$ is  compact, satisfies the L\"owenheim property, and has the property that the vocabulary of each sentence is finite, then any two disjoint $\ec{L}$-classes can be separated by an $\ec{\loo}$-class.
\end{theorem}
 
For Lindstr\"om Theorems for fragments of first order logic we refer to \cite{MR2529656}.
 
\section{Generalized quantifiers}\label{Generalized quantifiers}

Sometimes one may want a logic that can express some particular property, be it cardinality, cofinality, connectedness, well-foundedness, or some other structural property. One way to accomplish this is by means of the concept of a generalized quantifier, introduced in \cite{MR89816} and extended in \cite{MR244012}.  Generalized quantifiers allow one to add a particular feature to a logic. Unsurprisingly, there is no guarantee that such an addition results in a nice logic, e.g. a logic with interpolation. To obtain a nice logic one probably has to add many generalized quantifiers. But then another problem arises. Any logic (closed under substitution) can be represented as the result of adding a number of generalized quantifiers to first order logic. Since it is unlikely that every logic is ``nice'' in any sense, it makes sense to limit oneself to adding just one or finitely many generalized quantifiers. 
 
 The first generalized quantifier to be considered (in \cite{MR89816}) was $$Q_\alpha x\phi(x)\iff|\{a:\phi(a)\}|\ge\aleph_\alpha.$$ Here $\alpha$ is a fixed ordinal chosen in advance. The best-known cases are $\alpha=0$ or $\alpha=1$. 
The extension $L(Q_\alpha)$  of first order logic by this quantifier has the L\"owenheim property for $\alpha=0$ and is  countably compact for $\alpha=1$. The following simple result, due to Keisler, is from \cite{MR453484}:

\begin{theorem}\label{17}
$L(Q_\alpha)$ does not have the interpolation property.
\end{theorem}

\begin{proof}
Let $E$ be a binary relation symbol and $S,S'$  unary relation symbols. Let $\phi$ the conjunction of 
\begin{description}
\item[(1)] $\forall x(xEx)\wedge\forall x\forall y(xEy\to yEx)\wedge\forall x\forall y\forall z((xEy\wedge yEz)\to xEz))$.
\item[(2)] $\forall x\exists y(xEy\wedge S(y))$.
\item [(3)]$\forall x\forall y((S(x)\wedge S(y)\wedge x E y)\to x=y)$.
\item[(4)] $Q_\alpha xS(x)$.
\end{description}
Let $\psi$ the sentence
\begin{description}
\item[(5)] $\forall x\exists y(S'(y)\wedge xEy)\to Q_\alpha xS'(x)$.
\end{description}
  
 Clearly, $\phi\models\psi$. Suppose $\phi\models\theta$ and $\theta\models\psi$, where $\theta\in L(Q_1)$ has vocabulary $\{E\}$ only. Then $\theta$ says that $E$ is an equivalence relation with at least $\aleph_\alpha$ equivalence classes. But we now  use an  Ehrenfeucht-\fraisse\ game to show that such a $\theta$ cannot exist. In the Ehrenfeucht-\fraisse\ game of $L(Q_\alpha)$ on two models $\mm$ and $\mn$ players build a partial isomorphism between $\mm$ and $\mn$. Suppose a partial isomorphism $p$ has been built. In addition to the usual moves of the Ehrenfeucht-\fraisse\ game of first order logic,  player $I$ has the option of choosing a subset $X$ of cardinality $\ge \aleph_\alpha$ of one of the models, say $\mm$, after which player $II$  chooses a subset $Y$ of size $\ge \aleph_\alpha$ of the other model, in this case $\mn$. After this player $I$ chooses an element $y\in Y$ and player $II$ responds by choosing an element $x\in X$. The game continues now with the partial mapping $p\cup\{(x,y)\}$ while $X$ and $Y$ are abandoned. It is not hard to show that if $II$ has a winning strategy in such a game of any finite length, then the models are $L(Q_\alpha)$-equivalent (see e.g. \cite[Chapter 10]{MR2768176}). For the current case we can use a model $\mm$ which is an equivalence relation with $\aleph_\alpha$ classes, each of them of size $\aleph_\alpha$. For $\mn$ we can use an equivalence relation with $\aleph_0$ classes, each  of size $\aleph_\alpha$. The winning strategy of player $II$ is the following. Suppose player $I$ has played a set $X\subseteq M$, and player $II$ now has to choose the set $Y\subseteq N$. Let $A$ be the union of all equivalence classes of elements of $\dom(p)$, and $B$ the union of all equivalence classes of elements of $\ran(p)$. If $X\setminus A\ne\emptyset$, we let $II$ choose $N\setminus B$. This is clearly a move that keeps $II$ in the game. We can therefore assume w.l.o.g. that $X\subseteq A$. Since $X\setminus \dom(p)$ is infinite, there is $a\in\dom(p)$ such that $X\setminus \dom(p)$ meets the equivalence class of $a$. Now we let $II$ choose $Y$ to be what is in the equivalence class of $p(a)$ outside $\ran(p)$. This is clearly a winning strategy for any finite number of moves. This contradicts the fact that $\theta$ separates $\mm$ and $\mn$. 
\end{proof}
 
The proof also shows that even Souslin-Kleene Interpolation fails for $L(Q_\alpha)$. In fact, the proof shows\footnote{This is because the strategy of player $II$ works not only for any finite game but even for the game of length $\omega$.} that $\Delta(L(Q_\alpha))\not\le L_{\infty\omega}(Q_\alpha)$. The Beth theorem fails for  $L(Q_1)$ (\cite{MR317923}), but surprisingly, it is consistent, relative to the consistency of ZF, that the weak Beth property holds for $L(Q_1)$ (\cite{MR783593}). 

\begin{theorem}[\cite{MR250861}]\label{undefty}
$L(Q_0)$ does not have the weak Beth property.
\end{theorem}

\begin{proof}The crucial property of $L(Q_0)$ here is that it can express ``every natural number has only finitely many predecessors''.
Let $\phi\in L(Q_0)$ say of its models $(A, E, R)$  that $(A,E)$ is a model of a suitable finite part $T$ of $ZFC$ and either the natural numbers of $(A, E)$ have non-standard elements\footnote{I.e. elements which have infinitely many predecessors. This is where we use $Q_0$.} and $R = \emptyset$ or the natural numbers of $(A, E)$ are standard\footnote{I.e. they all  have only finitely many predecessors. This is again where we use $Q_0$.} and $R$ is the set of pairs $(\eta,f)$, where
$(A, E)$ satisfies 
\begin{enumerate}
\item $\eta \in L(Q_0)$, and 
\item $f$ is a function such that the inductive clauses for
satisfaction of $L(Q_0)$-formulae of the vocabulary $\{E\}$ hold.
\end{enumerate} An example of the inductive clauses here is:
\begin{itemize}
\item $(Q_0x \eta,f)\in R$  if and only if for infinitely many $y$
 there exists $g \in A$  such that $g(z) = f(y)$ for
variables $z \ne x$, $g(x)=y$ and $(\eta, g) \in R$.
\end{itemize}
lf $(A, E, R)$ and $(A, E, R')$ are models of $\phi$, and (w.l.o.g.) the integers of $(A,E)$ are all standard, then one can use induction on formulas of $L(Q_0)$ of the vocabulary $\{E\}$ to prove that $R = R'$.
Moreover, for all $(A, E)$ there is always an $R$ such that
$(A, E, R) \models\phi$. Suppose that there were a formula $\eta(x, y) \in L(Q_0)$ which defines $R$
explicitly in models of $\phi$. Let $\kappa>\omega$ be a  cardinal such that  $\mm = (H_\kappa, \in)$, where $H_\kappa$ is the set of sets whose transitive closure has cardinality $<\kappa$, is a model of $T$. If now $R$ is chosen such that
$(H_\kappa, \in,R)\models\phi$, then
$$R=\{(\phi,f):\phi\in L(Q_0)\mbox{ and $f$ satisfies $\phi$ in $\mm$}\}.$$
Combining this with the choice of $\eta(x, y)$ yields
$$\mm\models \eta(\phi,f)\mbox{ if and only if  $f$ satisfies $\phi$ in $\mm$}.$$
The standard diagonal argument ends the proof. To this end, let $\xi (= \xi(x))$ be the
formula $\neg\eta(x,f)$, where $f$ is a term denoting the function which maps the variable
$x$ to $x$ $(f = \{(x, x)\})$. We now have $\xi\in H_\kappa$
and
$$\mm\models\xi(\xi)\leftrightarrow\eta(\xi,\{(x,x)\})\leftrightarrow\neg\xi(\xi).$$
This contradiction shows that $\phi$ constitutes a counter-example to the weak Beth property for the logic $L(Q_0)$. 
\end{proof}

The point of the above proof is that with $L(Q_0)$ we can capture the standard natural numbers and thereby use induction on objects, which are formulas ``merely''  in the sense of the model $(A,E)$. Otherwise it does not matter what the new quantifier $Q_0$ says. 
The same proof gives a result about any generalized quantifiers $Q^1,\ldots,Q^n$ in the sense of \cite{MR244012}: If  $L(Q^1,\ldots,Q^n)\not\equiv\loo$ has L\"owenheim property, then the weak Beth property fails (\cite{MR244013}).

If we look how much we have to add to $L(Q_0)$ to obtain  weak Beth (or even interpolation), the optimal answer is $\wbeth(L(Q_0), L_{\mbox{\tiny HYP}})$, where ${\mbox{HYP}}$ is the smallest admissible set containing $\omega$ (see \cite{MR424560} for details) and $L_{\mbox{\tiny HYP}}=\looo\cap \mbox{HYP}$. Thus we have to resort to infinitary propositional operations to obtain Beth definability for $L(Q_0)$, and then we get full interpolation as a bonus.

If we want interpolation for a logic of the form $L(Q^1,\ldots,Q^n)$, we need to have countable (or at least ``recursive'') compactness.
However, no proper extension of first order logic is known which has both interpolation and countable compactness.
   
 \medskip
 
\noindent {\bf Problem:}
Are $FO$ quantifiers the {only}  generalized quantifiers  which give rise to a regular logic with  the Craig Interpolation property? The same for the Beth  property and Souslin-Kleene. Does weak Beth hold for $L(Q_1)$ provably in ZFC?
\medskip

\noindent {\bf Problem:}
 Is there any proper extension of first order logic that is (countably) compact and has the interpolation property?
\medskip


One of the most notorious generalized quantifiers is the cofinality quantifier $Q^{\mbox{\tiny cof}}_\kappa$ (\cite{MR376334}).  For a regular $\kappa$ it is defined as follows:
$$\begin{array}{lcl}
\mm\models Q^{\mbox{\tiny cof}}_\kappa xy\phi(x,y,\vec{a})&\iff&\{(c,d):\mm\models\phi(c,d,\vec{a})\}\\
&&\mbox{ is a linear order of cofinality $\kappa$.} 
\end{array}$$
What is remarkable about this quantifier is that  it is compact in a vocabulary of any cardinality, which is why it is called ``fully compact''.  It has also a nice complete axiomatization. However, the logic  $\lcof$ does not have interpolation or Beth properties (\cite{MR625528}). Still, interestingly, there is a countably compact logic $\laa$ with $\craig(\lcof,\laa)$, see Theorem~\ref{cof-aa}.


\section{Infinitary logics}\label{Infinitary logics}
 
Infinitary logics, such as $\looo$, are based on a completely different idea than generalized quantifiers. Here we add new logical operations\footnote{In $L_{\kappa\lambda}$ conjunctions and disjunctions of sets of formulas of size $<\kappa$ are allowed as well as quatification of sequences of variables of length $<\lambda$. Then $L_{\infty\lambda}$ is $\bigcup_\kappa L_{\kappa\lambda}$ and $L_{\infty\infty}$ is $\bigcup_\lambda L_{\infty\lambda}$.} which resemble the familiar logical operations of $\loo$ but are ``infinitary''.  While in the case of generalized quantifiers we did not mind if the meaning of formulas became ``infinitary'', as in ``there are infinitely many'' and ``there are uncountably many'', here we do not mind if the formulas themselves become syntactically ``infinitary'', as in $\phi_0\vee\phi_1\vee\ldots$. The advantage of this change of perspective is that we obtain lots of examples of extensions of first order logic with interpolation. Since we lose compactness (apart from  so-called Barwise compactness), we cannot obtain the Robinson property, the model theoretic version of interpolation. The first proofs of interpolation for infinitary logics were indeed proof-theoretic.

  
A paradigmatic example of an $\looo$-sentence is $\forall x\bigvee_n (x=s^n(0))$ which says of a unary function $s$ and a constant $0$ that every element is either $0$ or obtained from $0$ by iterating the function $s$, taking $s^0(0)$ to be $0$.
 
We embrace now fully the many-sorted approach to logic as it yields a particularly powerful form of interpolation. Thus we have variables $x^s$ of different sorts $s$. Let $\sort(\tau)$ denote the set of sorts of the vocabulary $\tau$. 
  
Let us recall the auxiliary concept of  a Hintikka set: Suppose $\tau$ is countable,
$C^s$ is a countable set of new constant symbols for $s\in\sort(\tau)$, and $\tau'=\tau\cup C^*$, where $C^*=\bigcup_sC^s$. A 
{\em {Hintikka set}} is any set $H$ of 
 $\tau'$-sentence of $\looo$, which satisfies:
 
\begin{enumerate}
\item $t=t\in H$ for every constant $\tau'$-term $t$. 

\item If $\phi(t)\in H$,  $\phi(t)$ atomic,  and
$t=t'\in H$, then $\phi(t')\in H$.

\item {If $\neg\phi\in H$, then 
$\phi\neg\in H$.\footnote{
$\phi\neg $ is $\neg\phi,\mbox{ if $\phi$ is atomic}$,
$(\neg\phi)\neg $ is $\phi$,
$(\bigwedge_n\phi_n)\neg  $ is $\bigvee_n\neg\phi_n$,
$(\bigvee_n\phi_n)\neg  $ is $\bigwedge_n\neg\phi_n$,
$(\forall x^{s}\phi)\neg  $ is $\exists x^{s}\neg\phi$, and
$(\exists x^{s}\phi)\neg  $ is $\forall x^{s}\neg\phi$
}}

\item {If $\bigvee_n\phi_n\in H$, then $\phi_n\in H$ for some $n$.}

\item {If $\bigwedge_n\phi_n\in H$, then $\phi_n\in H$ for all $n$.
}

\item {If $\exists x^{s}\phi(x^{s})\in H$, then
$\phi(c)\in H$ for some $c\in C^s$.}

\item If $\forall x^{s}\phi(x^{s})\in H$, then $\phi(c)\in H$
for all $c\in C^s$.

\item { For every constant $\tau'$-term $t$ of sort $s$ there is  $c\in C^s$  such that $t=c\in
H$.}

\item {There is no atomic  $\phi$  such that
$\phi\in H$ and $\neg\phi\in H$.}
\end{enumerate}
The Hintikka set $H$ is a Hintikka set {\em {for}} a sentence $\phi$ of
$\looo$ (or $\loo$) if $\phi\in H$.
The basic property of Hintikka sets is that if $\phi\in\looo$ does have  a model $\mm$, then there is a Hintikka set for $\phi$, built directly from formulas true in $\mm$; and conversely, if
there is a Hintikka set $H$  for a given $\phi\in\looo$, then $\phi$ has a model, as we shall now see. A model $\mm$ is built from $H$ as follows: Define on $C^*=\bigcup_s C^s$ an equivalence relation $c\sim c'$ by the condition $c=c'\in H$. Let 
the domain of sort $s$ in $\mm$ be $M_s=\{[c]:c\in C^s\}$. The interpretations in $\mm$ are defined by  
 $c^\mm=[c]$,
$f^\mm([c_{i_1}],\ldots,[c_{i_n}])=[c]$  for  $c\in C^*$
such that $f(c_{i_1},\ldots, c_{i_n})=c\in H$. 
 For any constant
term $t$ of sort $s$ there is a $c\in C^s$ such that $t=c\in H$. It is
easy to see that {$t^\mm=[c]$}. 
 If $R$ is an $n$-ary predicate symbol, we let $(t_1,\ldots,t_n)\in R^\mm$ if and only if $R(t_1,\ldots,t_n)\in H$. 
 By {induction} on $\phi(x_1,\ldots,x_n)$ one can now easily prove that if  $d_1\ldots,d_n\in C^*$ then, 
 $$\begin{array}{lcl}
\phi(d_1,\ldots,d_n)\in H&\Rightarrow& \mm\models\phi(d_1,\ldots,d_n)\\
\neg\phi(d_1,\ldots,d_n)\in H&\Rightarrow&\mm\not\models\phi(d_1,\ldots,d_n). 
\end{array}$$
In
particular, $\mm\models \phi$ for the $\phi$ we started with, since $\phi\in H$.

How do we {find} useful Hintikka sets?  The basic tool is the auxiliary concept of a {consistency property}.
Roughly speaking, a consistency property is a set  $\Delta$ of (usually) finite sets $S$ which are consistent and  $\Delta$ has information about how to extend $S$ to  a Hintikka set. 

 A {\em {consistency property}} is any set
$\D$ of countable sets $S$ of $\tau$-formulas of $\looo$, which
satisfies the conditions:

\begin{enumerate}
\item {If $S\in\D$, then $S\cup\{t=t\}\in \D$ for every
constant $\tau'$-term $t$.
}
\item {If $\phi(t)\in S\in\D$,  $\phi(t)$ atomic,  and
$t=t'\in S$, then $S\cup\{\phi(t')\}\in \D$.}

\item {If $\neg\phi\in S\in\D$, then
$S\cup\{\phi\neg\}\in \D$.}

\item {If $\bigvee_n\phi_n\in S\in\D$, then
$S\cup\{\phi_n\}\in \D$ for {some} $n$.}

\item {If $\bigwedge_n\phi_n\in S\in\D$, then
$S\cup\{\phi_n\}\in \D$ for {all} $n$.}

\item {If $\exists x^{s}\phi(x^{s})\in S\in\D$, then
$S\cup\{\phi(c)\}\in \D$ for {some} $c\in C^s$.}

\item {If $\forall x^{s}\phi(x^{s})\in S\in\D$, then
$S\cup\{\phi(c)\}\in\D$ for {all} $c\in C^s$.}

\item {For every constant $L'$-term $t$ of sort $s$ there is $c\in C^s$ such that
$S\cup\{t=c\}\in \D$.}

\item {There is no atomic formula $\phi$ such that $\phi\in
S$ and $\neg\phi\in S$.}
\end{enumerate}

The consistency property $\D$ is a consistency property {\em for}
a set $T$ of infinitary $\tau$-sentences if  for all $S\in\D$ and
all $\phi\in T$ we have $S\cup\{\phi\}\in \D$.

A $\tau$-{\em {fragment}} of $\looo$ is any
  set $\mathcal{F}$ of formulas of $\looo$ in the vocabulary $\tau$ such that
 $\mathcal{F}$ is closed under substitutions of terms,
$\mathcal{F}$ contains the atomic $\tau$-formulas,
 $\neg \varphi \in \mathcal{F}$ if and only if $\varphi \in \mathcal{F}$,
 $\wedge \Phi \in \mathcal{F}$ if $\Phi \subseteq \mathcal{F}$ is finite,
  $\vee \Phi \in \mathcal{F}$  if $\Phi \subseteq \mathcal{F}$, is finite, 
 $\wedge\Phi \in\mathcal{F}$ if and only if  $\vee \Phi \in \mathcal{F}$, and then $\Phi\subseteq\mathcal{F}$,
 $\forall x^{s} \varphi \in \mathcal{F}$ if and only if $\varphi \in \mathcal{F}$,
   $\exists x^{s} \varphi \in \mathcal{F}$ if and only if $\varphi \in \mathcal{F}$.
Note that a fragment is necessarily closed under subformulas. 
It is easy to see that if $\varphi \in \looo$ with a countable vocabulary $\tau$, then there
  is a countable fragment $\mathcal{F}\subseteq \looo$ such that $\varphi \in \mathcal{F}$.

\begin{lemma}\label{hintikka set lemma}
Let $T$ be a countable set of  $\tau$-sentences of $\looo$ or $\Loo$ and $\mathcal{F}\subseteq \looo$  a countable fragment such that $T\subseteq \mathcal{F}$.
 Suppose $\D$ is a consistency property for $T$.
Then for any $S\in \D$ there
is a Hintikka set $H$ for $T$ such that $S\subseteq H$.

\end{lemma}

{\bf Proof:} 
%
 We can define $H$ as the union of an increasing sequence $S_n$, $n<\omega$, where
 $S_0=S$. The sequence is constructed straightforwardly in such a  way, maintaining  judicious bookkeeping, that in the end the set $H$ is a Hintikka set. \qed

 Let $\un(\phi)$ be all sorts $s$ such that a variable of sort $s$ occurs universally quantified in $\phi$. Similarly $\un(S)$ for a set $S$ of formulas.
 Let $\ex(\phi)$ be all sorts $s$ such that a variable of sort $s$ occurs existentially quantified in $\phi$. Similarly $\ex(S)$ for a set $S$ of formulas.
For example,  suppose $\phi$ is
 $\forall x^{1}\exists x^{0}(x^0=x^1)$. Then $\un(\phi)=\{1\}$ and $\ex(\phi)=\{0\}$.
 Suppose $\phi$ is $\forall x^{0}\forall y^{3}(R(x^{0},y^{3})\leftrightarrow R'(x^{0},y^{3}))$. Then $\un(\phi)=\{0,3\}$ and $\ex(\phi)=\emptyset$.
 
\begin{theorem}[\cite{MR235996}]\label{craig1}
Suppose $\phi\models\psi$, where $\phi$ and $\psi$  are sentences of $\looo$ in a relational vocabulary. Then there is a
sentence $\theta$ of $\looo$ such that
\begin{enumerate}
\item $\phi\models\theta$ 
and $\theta\models\psi$
\item $\tau(\theta)\subseteq\tau(\phi)\cap\tau(\psi)$
\item $\un(\theta)\subseteq \un(\phi)$
and $\ex(\theta)\subseteq  \ex(\psi)$.
 
\end{enumerate}
\end{theorem}

Proof:  Let
us assume that the claim of the theorem is false and derive a
contradiction. Since $\phi\models\psi$, the set
 $\{\phi,\neg\psi\}$ has no models.
We construct a
consistency property for $\{\phi,\neg\psi\}$. It then follows that $\{\phi,\neg\psi\}$ has a model, which is a contradiction.
  Let $\tau_1=\tau(\phi)$, $\tau_2=\tau(\psi)$, and $\tau=\tau_1\cap \tau_2$.  Suppose $C^s=\{c^s_n:n\in\oN\}$ is a set of new
constant symbols for each sort $s$ of $\tau_1\cup \tau_2$. Let $C^*=\bigcup_sC^s$.
%
 Given a set $S$ of sentences, let {$S_1$} consists of all $\tau_1\cup C^*$-sentences in
$S$ with only finitely many constant from $C^*$, and let {$S_2$} consists of all $\tau_2\cup C^*$-sentences in $S$ with only finitely many constant from $C^*$. 
 Let
us say that $\theta$ \emph{separates} $S'$
and $S''$ if 
\begin{enumerate}
 
\item $S'\models\theta$, 
\item $S''\models\neg\theta$,  
\item $\un'(\theta)\subseteq \un(S')$, 
\item $\ex'(\theta)\subseteq\un(S'')$, 
\end{enumerate}where $\un'(\theta)$ consists of sorts $s\in\un(\theta)$ and sorts of constants $c\in C^*$ occurring in $\theta$, and $\ex'(\theta)$ consists of sorts $s\in\ex(\theta)$ and sorts of constants $c\in C^*$ occurring in $\theta$.

 Let $\D$ consist of finite sets $S$ of sentences of $\looo$ such that {$S=S_1\cup S_2$} and:
\begin{itemize}

\item[($\star$)] There is no $L\cup C^*$-sentence
that separates $S_1$ and $S_2$.
\end{itemize}
Note that $\{\phi,\neg\psi\}\in\D$. We show that $\D$ is a consistency property. This involves checking  the  
nine conditions that a consistency property has to satisfy. Most are almost trivial. For example, let us look at condition 7.
Consider $S\in\Delta$ and {$\exists x^{s}\phi(x^{s})\in S_1$}. Let $c_0\in C^s$ be such that
$c_0$ does not occur in $S$. Now the sets
$S_1\cup\{\phi(c_0)\}$ and $S_2$ satisfy ($\star$). 
%
%
On the other hand, consider $S\in\Delta$ and {$\exists x^{s}\phi(x^{s})\in S_2$}. Let again $c_0\in C^s$ be such that
$c_0$ does not occur in $S$. Now the sets
$S_1$ and $S_2\cup\{\phi(c_0)\}$ satisfy ($\star$). 
Theorem~\ref{craig1} is proved. \qed

If above $\phi,\psi\in A$, where $A$ is a countable admissible set, then also $\theta\in A$ \cite{MR424560}.

Many sorted  interpolation gives numerous {preservation} results of which we mention just one. It should be noted that we are dealing with infinitary logic, so we do not have a compactness theorem. Also, for example, it is not true that every formula has a prenex normal form (\cite{MR539973}). On the other hand,  Theorem~\ref{craig1}  gives also  many-sorted interpolation for first order logic, so it can be used to prove preservation results for first order logic as well.

Here is an example:

\begin{theorem}[\cite{MR245437}]
A formula $\phi$ of $\looo$ is preserved by submodels if and only if it is logically equivalent to a universal formula. 
\end{theorem}

\begin{proof} Let us assume the single sorted $\phi$ is written in sort $0$ variables and has just one binary predicate symbol $R$. Let $\phi'$ be the same written in sort $1$ variables and with $R$ replaced by $R'$. Let $\ext$ be the conjunction 
\begin{equation}\label{RR}
\forall x^{1}\exists x^{0}(x^0=x^1)\wedge\forall x^{1}\forall y^{1}(R'(x^{1},y^{1})\leftrightarrow R(x^{1},y^{1})).
\end{equation}
Note that $(\{M_0,M_1\},R,R')\models\forall x^{1}\exists x^{0}(x^0=x^1)$ iff $M_1\subseteq M_0$. The sentence (\ref{RR}) is true in $(\{M_0,M_1\},R,R')$ iff $(M_1,R')\subseteq(M_0,R)$. 
 By assumption, 
$\ext\wedge \neg\phi'\models\neg\phi$. Let $\theta\in\looo$  be given by  Theorem~\ref{craig1}. Thus
$\ext\wedge\neg\phi'\models\theta$ and  $\theta\models\neg\phi$.
The only common sort is $0$, so $\theta$ has only sort $0$ symbols. To see that $\theta$ is existential we note that if a sort $0$ variable was universally quantified in $\theta$, then it is universally quantified in $\ext\wedge\neg\phi'$, but there is no universally quantified sort $0$ variable in $\ext\wedge\neg\phi'$. Thus $\theta$ is existential. Moreover, $\models\neg\phi\leftrightarrow\theta$.
  \end{proof}
 

The infinitary logic $\looo$ can be extended by new logical operations preserving its ``good'' properties such as interpolation: Let us call any function $\P(\omega) \to 2$ a \emph{propositional connective}. If $P$ is a propositional connective, then $\looo(P)$ is the extension of $\looo$ by the connective
$$\mm\models_P C(\langle\phi_i:i<\omega\rangle)(\vec{a})\iff P(\{i:\mm\models_P\phi_i(\vec{a})\})=1.$$
Harrington \cite{MR603330} proved that there are $2^{2^\omega}$ propositional connectives $P$ such that $\looo(P)$ satisfies the Craig Interpolation Theorem as well as a form of the Barwise Compactness Theorem. Unfortunately Harrington's new propositional connectives are just abstract functions with no intuitive or natural meaning. Still their existence demonstrates that  $\looo$ is by no means maximal with respect to the Craig Interpolation Theorem, even if a weak form of compactness is added.

The situation changes radically when we move to $L_{\omega_2\omega}$:
  
\begin{theorem}[\cite{MR290943}]
$\craig(L_{\omega_2\omega},\lio)$ fails.
\end{theorem}

\begin{proof}
Let $\phi\in L_{\omega_2\omega}$ say  ``$<$ is a linear order of order-type $\omega_1$'' \cite{MR200133}. Let $\psi\in L_{\omega_1\omega}$ say ``$<'$ is a total linear order of order-type $\omega$''. Clearly $\phi\models\neg\psi$. Suppose $\theta\in \lio$ is such that $\phi\models\theta$, $\theta\models\neg\psi$ and the vocabulary of $\theta$ is $\emptyset$. Then $(\omega_1,<)\models\phi$, so $(\omega_1)\models\theta$. But $(\omega_1)\equiv_{\infty\omega}(\omega)$, as an easy application of the \ef\ game of length $\omega$ shows. So $(\omega)\models\theta$, a contradiction.
 \end{proof}

 The method of proof of Theorem~\ref{undefty}, i.e. essentially an undefinability of truth argument, can be used to prove:
 
\begin{theorem}[\cite{MR421982}]
\begin{enumerate}
\item $\wbeth(L_{\omega_2\omega}, L_{\omega_2\omega_2})$ fails.
\item $\wbeth(L_{\omega_1\omega_1},L_{\infty\infty})$ fails.
\end{enumerate}\end{theorem}  

However, with an appropriate modification of the method of  the proof of Theorem~\ref{craig1} one can show $\craig(L_{\kappa\omega},L_{(2^{<\kappa})^+\kappa})$, where $\kappa$ is regular (\cite{MR290943}).
By modifying the question of interpolation suitably, one can obtain a more balanced result, see \cite{MR1777793}.
The result $\craig(L_{\kappa\omega},L_{(2^{<\kappa})^+\kappa})$ raises the question, whether or not there is a logic $L$ such that 
$L_{\kappa\omega}\le L\le L_{(2^{<\kappa})^+\kappa}$ and $L$ has interpolation? Shelah's new infinitary logic $L^1_\kappa$ gives one answer to this question (\cite{MR2869022}). 

To see this, 
%
 %
let us redefine $L_{\kappa\omega}$  as $\bigcup_{\lambda<\kappa}L_{\lambda^+\omega}$
and $L_{\kappa\kappa}$ as $\bigcup_{\lambda<\kappa}L_{\lambda^+\lambda^+}$ in the
case that $\kappa$ is a limit cardinal. For regular limits this agrees
with the old notation for $L_{\kappa\omega}$ and $L_{\kappa\kappa}$, but for
singular cardinals the new notation seems more canonical. Let
$\kappa=\beth_\kappa$, where $\beth_0=\omega$, $\beth_{\alpha+1}=2^{\beth_\alpha}$ and $\beth_\nu=\sup\{\beth_\alpha:\alpha<\nu\}$ for limit ordinals $\nu$. The new infinitary logic $\lkp$ introduced in
\cite{MR2869022} satisfies $L_{\kappa\omega}\le \lkp$, $\lkp\le L_{\kappa\kappa}$
and $\craig(\lkp)$. Moreover, $\lkp$ has a Lindstr\"om style model
theoretic characterization in terms of a strong form of
undefinability of well-order. 

The main ingredient of $\lkp$ is a
new variant  $G^\beta_\theta(A,B)$, where $\beta$ is an ordinal and $\theta$ a cardinal, of the Ehrenfeucht-\fraisse\ game for $L_{\kappa\kappa}$. In
this variant player $\II$ gives only partial answers to moves
of player $\I$. She makes promises and fulfills the promises move by move. 
If the game lasted for $\omega$ moves, the partial
answers of $\II$ would constitute full answers. But the game has a
well-founded clock, so $\II$ never ends up fully completing her
answers. Although $\lkp$ cannot express well-ordering, it can
express e.g. the property of a linear order of not having an uncountable
descending chain.

Shelah's game $G^\beta_\theta(A,B)$ proceeds as follows: 
\begin{itemize}
\item At first player I picks $\beta_0<\beta$ and  $\vec{a^0}$  from $A$ with $len(\vec{a^0})\le\theta$. The move $\beta_0$ is a ``clock'' move which controls the (finite) length of the game. The move $\vec{a^0}$ is player I's challenge to player II. He wants to know what are the images of the $\le \theta$ elements of the sequence $\vec{a^0}$. 
\item Next II picks $f_0:\vec{a^0}\to\omega$ and  $g_0:A\to B$ a partial isomorphism such that  $f_{0}^{-1}(0)\subseteq\dom(g_{0})$. Here player II responds to the challenge made by player I. But now comes the catch. Player II divides the set $\vec{a^0}$ into $\omega$ pieces  with $f_0$. She is not going to respond yet to the entire challenge $\vec{a^0}$, only to a piece, namely  $f_{0}^{-1}(0)$. To the other pieces she is going to respond later, one by one. With good luck the game ends before she has to respond to $f_{0}^{-1}(10^{10})$! 
\item Next I picks 
$\beta_1<\beta_0$ and $\vec{b^1}$ from $B$ such that $len(\vec{b^1})\le\theta$. Here I challenges II to give the preimages of the $\le\theta$ elements of $\vec{b^1}$.
\item Next II picks 
$f_1:\vec{b^1}\to\omega$ and $g_1:A\to B$ a partial isomorphism, $g_1\supseteq g_0$ such that $f_{0}^{-1}(1)\subseteq\dom(g_{1})$ and $f_{1}^{-1}(0)\subseteq\ran(g_{1})$.  With $f_1$ player II splits the challenge $\vec{b^1}$ into $\omega$ pieces to which she is going to respond one by one while the game continues. With $g_1$ she starts responding to the challenge 
$\vec{b^1}$. In order for her responses to remain consistent, it is necessary that $g_1\supseteq g_0$. By making sure that  $f_{1}^{-1}(0)\subseteq\ran(g_{1})$ she responds to the challenge $f_{1}^{-1}(0)$. Now she has to also continue responding to the challenge $\vec{a^0}$ by making sure $f_{0}^{-1}(1)\subseteq\dom(g_{1})$.
\item And so on until $\beta_n=0$.
\end{itemize}
Player II {wins} if she can play all her moves, otherwise Player I wins.
 $A\sim^\beta_\theta B$ if Player II has a winning strategy in the game.
 $A\equiv^\beta_\theta B$ is defined as the transitive closure of $A\sim^\beta_\theta B$.
 A union of $\le \beth_{\beta+1}(\theta)$ equivalence classes of $\equiv^\beta_\theta$ for some $\theta<\kappa$ and $\beta<\theta^+$ is called a {sentence} of $L^1_\kappa$.
 
How does this definition make $L^1_\kappa$ an abstract logic? One has to work a bit to prove that all the closure properties that we require of a logic are satisfied. In the end, $L^1_\kappa$ is an abstract logic in the sense of Definition 1.1.1 of \cite{MR819531}. But what is the syntax like? All we know about the sentences of $L^1_\kappa$ is that they are classes of models. Still we can prove:

\begin{theorem}[Shelah \cite{MR2869022}]
Suppose $\kappa=\beth_\kappa$. Then $L^1_\kappa$ satisfies the Craig Interpolation Theorem.
\end{theorem}
 
The proof is based entirely on properties of the game   $G^\beta_\theta(A,B)$ and is not totally unlike the proof of Theorem~\ref{lindstr}. It is somewhat surprising that many model theoretic results can be proved for $L^1_\kappa$ although we are not able (yet) to answer the question:
 
 \medskip
 
\noindent {\bf Problem:}
  What is the syntax of Shelah's logic? 
\medskip
 
  Cartagena logic \cite{kivimäki2024cartagenalogic} is a syntactic fragment of Shelah's logic. Its $\Delta$-extension is Shelah's logic $L^1_\kappa$. For more on $L^1_\kappa$, see  \cite{boban}.
  

\section{Higher order logics}\label{Higher order logics}


Higher order logic are the oldest extensions of first order logic. 
In \emph{weak second order logic} $L^2_{\mbox{\tiny w}}$ there are variables for individuals, as in first order logic, but also variables for finite sets of individuals. If $X$ is a set-variable and $t$ is a term denoting an individual, then we  can form the atomic formula $X(t)$. Quantifiers range over individual and over set variables. The logic $L^2_{\mbox{\tiny w}}$ is quite close to the logic $L(Q_0)$ (see section~\ref{Generalized quantifiers}), in fact, $\Delta(L^2_{\mbox{\tiny w}})\equiv\Delta(L(Q_0))$.

It was shown in \cite{MR250861} that the Beth property fails for \emph{weak} second order logic, and for  \emph{monadic} second order logic in which second order variables range over arbitrary subsets of the domain but vocabularies are canbe non-monadic. The former claim is proved by means of an undefinability of truth argument, almost verbatim as we did in the  proof of Theorem~\ref{undefty}, while the latter claim is proved by reduction to  the decidable monadic second order theory of the successor function (\cite{MR183636}). Namely, in the standard model of the successor function, which is definable in monadic second order logic, addition and multiplication are implicitly definable by their familiar recursive definitions. If they were explicitly definable, arithmetic would be reducible to the monadic second order theory of the successor function. However, this is impossible because the former is non-arithmetic but the latter is decidable by \cite{MR183636}.


In {stationary logic} $\laa$, introduced in \cite{MR376334} and \cite{MR486629}, we have variables for individuals and also for countable sets of individuals.  If $s$ is a set-variable and $t$ is a term denoting an individual, then we  can again form the atomic formula $s(t)$. This time we do not have existential and universal quantifiers for set-variables. Instead we have a quantifier which can say that ``most'' countable sets have some property. What does ``most'' mean when we talk about countable sets? Here we use the concept of a club set. A set  $C$ of countable subsets of $M$ is \emph{unbounded} if for every countable $X\subseteq M$ there is $Y\in C$ such that $X\subseteq Y$. The set $C$ is \emph{closed} if it is closed under unions of increasing $\omega$-chains. A \emph{club} set means a closed unbounded set. The club (countable) subsets of a set $M$ form a (normal) filter\footnote{For the definitions of filter and ultrafilter, see \cite{MR1940513}.}, which however is not an ultrafilter. Still, it is a useful measure of largeness. We adopt the following generalized second order quantifier:

$$\begin{array}{lcl}
\mm\models aa s\phi(s,\vec{a})&\iff&\{A\subseteq M\ :\ |A|\le\omega, (\mm,A)\models\phi(A,\vec{a})\}\\
&&\mbox{ contains a club of countable subsets of $M$}. 
\end{array}$$

As proved in \cite{MR376334} and \cite{MR486629}, the logic $\laa$ is countably compact and has a nice complete axiomatization. Of course, it does not have the L\"owenheim property, but every consistent sentence has a model of cardinality at most $\aleph_1$. It extends both $L(Q_1)$ and $\lcof$.
The logic $\laa$ does not have Beth property \cite{MR625528}, but:

\begin{theorem}[\cite{MR808815}]\label{cof-aa}
$\craig(\lcof,\laa)$.
\end{theorem} 

\begin{proof}(A rough sketch) We actually prove $\rob(\lcof,\laa)$. 
%
This involves proving that if $T_1$ and  $T_2$ are countable complete $\laa$-theories, $T_0=T_1\cap T_2$ is complete, and $\tau(T_1)\cap \tau(T_2)=\tau(T_0)$, then 
$(T_1\cap\lcof)\cup(T_2\cap\lcof)$ is consistent. Theories $T_1$ and $T_2$ are first enriched by new unary predicates $P_\alpha$, $\alpha<\omega_1$. We add the axiom $\phi(P_{\alpha_1},\ldots,P_{\alpha_n})$ to $T'_l$, whenever $\alpha_1<\ldots<\alpha_n<\omega_1$ and $aa s_1\ldots aa s_n\phi(s_1,\ldots,s_n)\in T_l$. These predicates build a canonical club in our final model giving us control of $\laa$-truth. 
To see how this works in a simple  case, consider the set of 
\begin{equation}\label{aap}
 \forall x_1\ldots\forall x_n((P(x_1)\wedge\ldots\wedge P(x_{n+1}))\to\phi^{(P)}(x_1,\ldots,x_n)),
\end{equation}
 where 
\begin{equation}\label{aa}
aa s \forall x_1\ldots\forall x_n((s(x_1)\wedge\ldots\wedge s(x_{n+1}))\to\phi^{(s)}(x_1,\ldots,x_n))\in T_1
\end{equation} and $\phi(x_1,\ldots,x_n)$ is first order. Since every model has a closed unbounded set of countable subsets that are domains of elementary submodels with respect to first order logic,\footnote{One first forms Skolem-functions and then takes submodels that are closed under the Skolem-functions. The collections of such submodels is always closed unbounded.} and the club sets form a filter, the collection of all sentences (\ref{aap}) essentially ``says'' in a model $\mm$ that $P^\mm$ is the domain of a countable elementary submodel of $\mm$.

Next $T'_1$ and $T'_2$ are reduced to complete first order theories $T^*_1$, $T^*_2$ by introducing new relation symbols, namely an $n$-ary symbol $R_\phi$ for each formula $\phi(x_1,\ldots, x_n)$ of $\laa$, 
with the defining axiom $$\forall x_1\ldots x_n(R_{\phi}(x_1,\ldots,x_n)\leftrightarrow\phi(x_1,\ldots,x_n)),$$ and taking only the resulting first order formulas. 

Then we use the Robinson property of first order logic to obtain a model $\mm$ of $T^*_1\cup T^*_2$. W.l.o.g. $\mm$ is $\aleph_2$-saturated (see e.g. \cite[p. 480]{MR1221741} for the definition and basic properties of $\aleph_2$-saturation). Let $\mn$ be the union of the submodels of $\mm$ determined by the predicates $P_\alpha$. Now we can show that if $\mn$ thinks the cofinality of a definable linear order is uncountable, it is in $\mn$ exactly $\aleph_1$. On the other hand, if $\mn$ thinks the cofinality of a definable linear order is $\omega$, the real cofinality of the linear order  in $\mn$ is actually $\aleph_2$. The final step is a chain argument, using  a chain of length $\omega$, to turn cofinality $\aleph_2$ to cofinality $\aleph_0$ without losing the cofinalities that are $\aleph_1$. 
 \qed

Why the above result is remarkable is that it is the closest we have got so far in obtaining interpolation among countably compact proper extensions of $\loo$. The situation raises the following intriguing question:

 \end{proof}

\medskip
 
\noindent {\bf Problem:}
 Is there a logic $L$ such that $\lcof\le L\le \laa$ and $L$ has the Craig interpolation property?
\medskip

Note that if a logic $L$ as in the above problem exists, then $L$ is also countably compact.

In full single-sorted second order logic $L^2$, with variables for arbitrary relations,  interpolation (even uniform\footnote{See \cite{chapter:predicate}}) {holds} trivially, because we can quantify over relations. The separation property, in particular, becomes trivially true as single-sorted $\pc{\sol}$-classes are always $\ec{\sol}$-classes. In particular, $\Delta^1_1(\sol)\equiv\sol$.
  In many-sorted second order logic interpolation  (even the weak Beth property) {fails}  because we can use new sorts to define truth implicitly, but not 
  explicitly (\cite{MR195685}, \cite{MR421982}), imitating the proof of Theorem~\ref{undefty}. The difference between the single-sorted and the many-sorted case is that in the latter the predicates that one would like to quantify away may involve new sorts, i.e. elements outside the current model. There is no way $\sol$ can reach out to them. In particular, $\Delta(\sol)\not\equiv\sol$.

\medskip
 
\noindent {\bf Problem:}
 Is there any (set-theoretically definable) extension of second order logic that has many-sorted interpolation?
\medskip
 
If we drop the requirement of set-theoretical definability, there is a  solution: sort logic (\cite{MR3205075}). If we just want 
a proper extension of first order logic with many-sorted interpolation, the logic $\looo$ offers a solution, as we have seen.


Existential second order logic $\eso$ (i.e. single-sorted $\pc{\loo}$) is a perfectly nice abstract logic, only it is not closed under negation. It satisfies 
(even uniform) single-sorted Craig Interpolation Theorem, as we can quantify over relations, i.e. in the single-sorted context $\pc{\eso}\equiv\eso$. For the same reason $\eso$ satisfies the Souslin-Kleene Interpolation 
and Beth Theorems. In each case the claim is trivially true. The usual proof of the Robinson property works also for $\eso$. 
The logic $\eso$ satisfies also compactness and has the L\"owenheim property. This does not violate Lindstr\"om's Theorem, because $\eso$ is not closed under negation. In \cite{MR4535931} it is proved that there is no strongest abstract logic without negation which is compact and satisfies the L\"owenheim property. Many-sorted $\pc{\loo}$ is actually equivalent to  single-sorted $\pc{\loo}$ (\cite{zbMATH03216175}). Thus also  many-sorted $\eso$ satisfies interpolation and the related properties following from interpolation.

Dependence logic $\dep$, based on  the  atom $=\!\!\!(\vec{x},\vec{y})$, which says that $\vec{x}$ totally determines $\vec{y}$, was introduced in  
\cite{MR2351449}. On the level of sentences $\dep$ is equivalent to $\eso$, hence it satisfies (even uniform) Craig Interpolation, Souslin-Kleene Interpolation, Beth theorem, and the Robinson property, whether single- or many-sorted.
 %
%
The Craig Interpolation property holds for $\dep$ formulas also in the following form: Suppose $\phi(\vec{x},\vec{y})\models\psi(\vec{x},\vec{z})$, where $\vec{y}\cap\vec{z}=\emptyset$. Then there is $\theta(\vec{x})$ such that 
$\phi(\vec{x},\vec{y})\models\theta(\vec{x})$ and $\theta(\vec{x})\models\psi(\vec{x},\vec{z})$. Namely, we can take $\theta(\vec{x})$ to be $\exists \vec{y}\phi(\vec{x},\vec{y})$. This works, because of locality\footnote{Locality of a formula means that the truth of the formula depends  only on values of assignments on variables which are free in the formula.} (\cite[Lemma 3.27]{MR2351449}). 

Inclusion Logic  $\inc$  (\cite{MR2846029}), based on  atoms of the form  $\vec{x}\subseteq\vec{y}$, which say that every value of $\vec{x}$ occurs as a value of $\vec{y}$, is not equal to $\eso$ on the level of sentences. Rather, inclusion logic is equivalent to Positive Greatest Fixed Point Logic $\mbox{GFP}^+$, i.e. the fragment  of Greatest Fixed Point Logic in which fixed point operators occur only positively (\cite{MR3111746}). Thus we cannot use reduction to $\eso$ to infer any interpolation properties for inclusion logic.

\medskip
 
\noindent {\bf Problem:}
 Does Inclusion Logic  $\inc$  (\cite{MR2846029}) satisfy  the Craig Interpolation property?
\medskip
 

\section{Conclusion}

  Interpolation in all its forms mostly fails in extensions of first order logic. It seems difficult to extend first order logic in a way which leads to the kind of balance required by interpolation and the Beth property. There is the artificial way of using the $\Delta$-operation to obtain $\Delta(L)$, which always has the Souslin-Kleene Interpolation property and therefore the weak Beth property. But there is no general method to find a syntax for the semantically defined $\Delta(L)$. The infinitary logic $\looo$ is remarkable in the richness of its model theory and, as we have seen, it has the interpolation property. In bigger infinitary logics interpolation systematically fails, signalling, perhaps, that there is something incomplete in their syntax. The new large infinitary logic $L^1_{\kappa}$ enjoys interpolation, but lacks so far a satisfactory syntax. Overall, the area abounds with open problems, only some of which have been mentioned above.

\bibliographystyle{plain}
\bibliography{Vaananen_for_Craig,taci}
\end{document}